\documentclass{amsart}
\numberwithin{equation}{section}
\usepackage{amsmath}
\usepackage{amssymb}
\usepackage{amsthm}
\usepackage[all]{xypic}
\usepackage{amsfonts}
\usepackage{fullpage}
\usepackage{verbatim}
\usepackage{url}
\usepackage{hyperref}

\newtheorem{lemma}{Lemma}[section]
\newtheorem{proposition}[lemma]{Proposition}
\newtheorem{theorem}[lemma]{Theorem}

\theoremstyle{definition}

{

}

\theoremstyle{remark}

\newtheorem{remark}[lemma]{Remark}

\newdimen\uboxsep \uboxsep=1ex
\def\uboxn#1{\vtop to 0pt{\hrule height 0pt depth 0pt\vskip\uboxsep
\hbox to 0pt{\hss #1\hss}\vss}}

\def\uboxs#1{\vbox to 0pt{\vss\hbox to 0pt{\hss #1\hss}
\vskip\uboxsep\hrule height 0pt depth 0pt}}

\newcommand{\DMO}{\DeclareMathOperator}

\newcommand{\beq}{\begin{equation}}
\newcommand{\eeq}{\end{equation}}

\newcommand{\ang}[1]{\langle #1 \rangle}

\newcommand{\st}{\left\vert\right.}

\newcommand{\ssm}{\smallsetminus}

\DMO{\id}{{Id}}
\DMO{\Hom}{{Hom}}

\DMO{\End}{{End}}
\DMO{\Ext}{{Ext}}
\DMO{\ext}{{ext}}
\DMO{\Tor}{Tor}
\DMO{\Aut}{{Aut}}
\DMO{\Proj}{Proj}
\DMO{\Spec}{Spec}
\DMO{\codim}{codim}
\DMO{\reg}{reg}
\DMO{\len}{len}
\DMO{\red}{red}
\DMO{\pd}{pd}
\DMO{\hd}{hd}
\DMO{\Ker}{Ker}
\DMO{\coker}{Coker}
\DMO{\im}{Im}
\DMO{\cohdim}{cd}
\DMO{\chrr}{char}
\DMO{\Skl}{Skl}
\DMO{\Inn}{Inn}
\DMO{\GK}{GKdim}
\DMO{\gr}{gr}
\DMO{\Div}{Div}
\DMO{\Bs}{Bs}
\DMO{\sing}{sing}
\DMO{\Qgr}{Qgr}
\DMO{\op}{op}
\DMO{\Ann}{Ann}
\DMO{\sat}{sat}
\DMO{\grade}{grade}
\DMO{\GKdim}{GKdim}
\DMO{\hilb}{hilb}
\DMO{\gldim}{gldim}
\DMO{\Tot}{Tot}
\DMO{\Pic}{Pic}
\DMO{\Bir}{Bir}

\newcommand{\mc}{\mathcal}

\newcommand{\mb}{\mathbb}

\renewcommand{\AA}{{\mb A}}

\newcommand{\NN}{{\mb N}}
\newcommand{\PP}{{\mb P}}

\newcommand{\ZZ}{{\mb Z}}

\newcommand{\kk}{{\Bbbk}}

\DMO{\shHom}{\mathcal{H}\!\mathit{om}}
\DMO{\shExt}{\mathcal{E}\!\mathit{xt}}
\DMO{\shTor}{\mathcal{T}\!\mathit{or}}

\newcommand{\sL}{\mc{L}}

\DMO{\rgr}{gr-\!}
\DMO{\lgr}{\!-gr}
\DMO{\lGr}{\!-Gr}
\DMO{\rGr}{Gr-\!}
\DMO{\rmod}{mod-\!}
\DMO{\lmod}{\!-mod}
\DMO{\lMod}{\!-Mod}
\DMO{\rMod}{Mod-\!}
\DMO{\mmod}{mod}
\DMO{\rmodu}{mod^{u}-\!}
\DMO{\lmodu}{\!-mod^{u}}
\DMO{\lQgr}{\!-Qgr}
\DMO{\lqgr}{\!-qgr}
\DMO{\lProj}{\!-Proj}
\DMO{\rQgr}{Qgr-\!}
\DMO{\rqgr}{qgr-\!}
\DMO{\rProj}{Proj-\!}
\DMO{\rTors}{Tors-\!}
\DMO{\lTors}{\!-Tors}
\DMO{\rtors}{tors-\!}
\DMO{\ltors}{\!-tors}
\DMO{\Xyz}{Xyz}
\DMO{\xyz}{xyz}
\DMO{\Bl}{Bl}

\newcommand{\ol}{\overline}
\newcommand{\fk}{\kk}

\title{Generalised Witt algebras and idealizers}
\author{S. J. Sierra and \v{S}. \v{S}penko}
\date{\today}



\begin{document}

\begin{abstract}
Let $\kk$ be an algebraically closed  field of characteristic zero, and let $\Gamma$ be an additive subgroup of $\kk$. 
Results of Kaplansky-Santharoubane  and Su classify  intermediate series representations of the {\em generalised Witt algebra} $W_\Gamma$ in terms of three families, one parameterised by $\AA^2$ and two by $\PP^1$.  
In this note, we use the first family to construct a homomorphism $\Phi$ from the enveloping algebra $U(W_\Gamma)$ to a skew extension $\kk[\AA^2]\rtimes \Gamma$ of the coordinate ring of $\AA^2$. We show that the image of $\Phi$ is contained in a (double) idealizer subring of this skew extension and  that the representation theory of idealizers  explains the three families. 
We further show that the image of $U(W_\Gamma)$ under $\Phi$ is not left or right noetherian, giving a new proof that $U(W_\Gamma)$ is not   noetherian.

We construct $\Phi$ as an application of a general technique to create ring homomorphisms from shift-invariant families of modules.
Let $G$ be an arbitrary group and let $A$ be a $G$-graded ring.  A graded $A$-module $M$ is an {\em intermediate series} module if $M_g$ is one-dimensional for all $g \in G$.  
Given a shift-invariant family of intermediate series $A$-modules parametrised by a scheme $X$, we construct a homomorphism $\Phi$ from $A$ to a skew extension of $\kk[X]$.  
The kernel of $\Phi$ consists of those elements which annihilate all modules in $X$.  
\end{abstract}

\maketitle
\section{Introduction}

Fix an algebraically closed  ground field $\kk$ of characteristic zero, and 
let $\Gamma$ be a finitely generated additive subgroup of $\kk$.  
The {\em generalised Witt algebra} $W_\Gamma$ is the Lie algebra generated by elements $e_\gamma: \gamma \in \Gamma$, 
with $[e_\gamma, e_\delta] = (\delta-\gamma) e_{\delta+\gamma}$.
Recall that an {\em intermediate series representation } of $W_\Gamma$ is an indecomposable representation all of whose $e_0$-eigenspaces are 1-dimensional. 
It is a theorem of Kaplansky and Santharoubane \cite{KaplanskySantharoubane} (if $\Gamma = \ZZ$) and of Su \cite{SuVirasoro} (for general $\Gamma$) that intermediate series representations of $W_\Gamma$ come in three families (with two modules represented twice):  one family parameterised by $\AA^2$ and two parameterised by $\PP^1$.
In this note we use the first family to construct a homomorphism $\Phi$ from $U(W_\Gamma)$ to $T = \kk[\AA^2] \rtimes \Gamma$, and show  that the existence of the other two families is a consequence of  the fact that the image of $U(W_\Gamma)$ is a sub-idealizer in $T$.
We further use the homomorphism $\Phi$ to give a new proof that the enveloping algebra of $U(W_\Gamma)$ is not noetherian, a fact originally proved in \cite{SierraWalton1}.

Since our main method is to construct and then analyze a homomorphism from $U(W_\Gamma)$ to an idealizer in $T$, 
we recall some facts about idealizers.
We first define $T$:  as a vector space  we write $T = \bigoplus_{\gamma \in \Gamma} \kk[a,b] t^\gamma$, with $t^\gamma t^\delta = t^{\gamma+\delta}$ and $t^\gamma f (a,b)= f(a+\gamma, b) t^\gamma = : f^\gamma t^\gamma$.  
Note that $T$ is a bimodule over $\kk[a,b]$. 

An {\em intermediate series} module $M$ over a $\Gamma$-graded ring is an indecomposable $\Gamma$-graded module with each $M_\gamma$ a one-dimensional vector space.  
It is a generalisation of a point module over an $\NN$-graded ring, which  is a cyclic graded module with Hilbert series $1/(1-t)$.

For $p = (\alpha, \beta) \in \AA^2$, let $I(p) $ be the ideal $ (a-\alpha, b-\beta) $ of $ \kk[a,b]$.  
Let  $V(p) = T/I(p)T$.   It is easy to see that the $V(p)$ are all of  the intermediate series right $T$-modules; more precisely, the right ideals $J$ of $T$ such that $T/J$ is an intermediate series module are precisely the $ I(p)T$.  
Likewise, the  intermediate series left $T$-modules are the $T/TI(p)$.  
These families are preserved under degree shifting.

We now consider a subring of $T$.  
Fix $p_0 \in \AA^2$, and let $S= S(p_0) = \kk \oplus  I(p_0) T$.
The ring $S$ is an {\em idealizer} in $T$:  the largest subalgebra of $T$ such the right  ideal $ I(p_0)T$ becomes a two-sided ideal in $S$.
It is known \cite{Rogalski} that the representation theory of idealizers  involves blowing up.
Here for  $p \neq p_0$ we have that $V(p) \cong S/(S\cap  I(p)T )$ is  an intermediate series right $S$-module.
On the other hand, to define an intermediate series right $S$-module at $p_0$, we need to consider a point $q$ {\em infinitely near} to $p_0$:  that is, an ideal $I(q)$ with $I(p_0)^2 \subseteq I(q) \subseteq I(p_0)$ of $\kk[a,b]$ such that $I(p_0)/I(q)$ is one-dimensional.  
Such ideals are parameterised by the exceptional $\PP^1$ in the blowup $\Bl_{p_0}(\AA^2)$; more specifically, we can write 
 $I(q) = (y(a-\alpha)-x(b-\beta),(a-\alpha)^2,(a-\alpha)(b-\beta),(b-\beta)^2)$ for some $[x:y] \in \PP^1$.
For such $I(q)$ we have that  $I(p_0) +  I(q)T$ is a right ideal of $S$.
Let 
\[P(q) = S/(I(p_0)+ I(q)T).\] 
Then $P(q)$ is a  intermediate series right $S$-module.  
In fact, we have  constructed all right ideals $J$ of $S$ such that $S/J$ is an intermediate series $S$-module; they are parameterised by $\Bl_{p_0}(\AA^2)$ but it is sometimes more convenient to consider them as parameterised by $\AA^2\ssm \{ p_0 \} $ together with   $\PP^1$.

{\em Left} intermediate series $S$-modules are also parameterised by $\Bl_{p_0}(\AA^2)$. 
For $p \in \AA^2\ssm \{ p_0 \} $, the left intermediate series module $T/TI(p)$ is isomorphic to  
$\left(I(p_0)\oplus \bigoplus_{0\neq \nu\in \Gamma}\fk[a,b]t^{\nu}\right)/\left((I(p_0)\cap I(p))\oplus \bigoplus_{0\neq \nu\in \Gamma}t^\nu I(p)\right)$.
  We can extend this construction to a family of modules  parameterised by $\Bl_{p_0}(\AA^2)$ by adding the $\PP^1$ of points $q$ infinitely near to $p_0$: 
\[
Q(q)= \frac{I(p_0)\oplus \bigoplus_{0\neq \nu\in \Gamma}\fk[a,b]t^{\nu}}{I(q)\oplus \bigoplus_{0\neq \nu\in \Gamma}t^\nu I(p_0)}.
\] 

Consider now right intermediate series modules over the double idealiser 
\[ R = \kk[a,b] + (I(p_0)T \cap T I(p_1))\]
 and assume for simplicity  that $p_0, p_1 \in \AA^2$ have distinct $\Gamma$-orbits.
These correspond to points of the double blowup $\Bl_{p_0,p_1}(\AA^2)$.
More precisely, the $V(p)$ are intermediate series modules for $p \in \AA^2 \ssm\{p_0, p_1\}$.
From the inclusion $R \subseteq \kk \oplus I(p_0)T$ we obtain a family $P(q)$ parameterised by the $\PP^1$ of points infinitely near to $p_0$.
Finally, from the inclusion $R \subseteq \kk \oplus TI(p_1)$ we obtain a family $Q(q)$ of right modules parameterised by the $\PP^1$ of points infinitely near to $p_1$ and constructed similarly to the construction of the left modules $Q(q)$ over $S$.  

Let $\Gamma$ now be an arbitrary group (more generally, a monoid) and let $A$ be a $\Gamma$-graded ring. 
We give a general result in Theorem \ref{thm:right} (respectively, Theorem \ref{thm:left}) which constructs a ring homomorphism (respectively, an anti-homomorphism) $\Phi:  A\to \kk[X] \rtimes \Gamma$, where $X$ is a  shift-invariant family of right (respectively, left) intermediate series $A$-modules; this generalises constructions in \cite{ATV1991,RZ2008,VdBtranslation}.

When we apply this technique to $U(W_\Gamma)$, we show that the image of $\Phi$ is contained in a double idealizer $R$ inside the ring $T$ defined in the second paragraph, and we show in Propositions \ref{prop:P}, \ref{prop:Q} that the right intermediate series $R$-modules constructed above restrict to precisely the intermediate series representations of $W_\Gamma$.
This gives a unified geometric description of what have until now been seen as three distinct families of representations.

We further show in  Proposition \ref{prop:noetherian} that the image of $U(W_\Gamma)$ {}under $\Phi$ is neither right nor left noetherian.
For $\Gamma = \ZZ$ this was proved in \cite{SierraWalton2} as the main step in proving the non-noetherianity of $U(W)$.
 It follows that $U(W_\Gamma)$ is neither right or left noetherian; other proofs are given in  \cite{SierraWalton1,SierraWalton2}.  

 The general behaviour of  idealizers leads one to expect  that at idealizers in $T$ at ideals of points in $\kk[a,b]$ will not be noetherian since no points have dense $\Gamma$-orbits; see \cite{Sierra1} for a precise statement of a related result for $\NN$-graded rings. 
However, infinite orbits are dense in $\AA^1$.  Thus one expects that the factors $\Phi(U(W_\Gamma))|_{b=\beta}$, which live on the $\Gamma$-invariant line $(b=\beta)$ in $\AA^2$,  are noetherian for all $\beta \in \kk$, and we also show in Proposition \ref{prop:noetherian2} that this is indeed  the case.

{\bf Acknowledgements: }  
 We thank Jacques Alev and Lance Small for useful discussions.
The second author is supported, and the first author  is partially supported, by 
 EPSRC grant EP/M008460/1.

\section{Intermediate series modules and ring homomorphisms}\label{sec1}

It is well-known that ring homomorphisms can be constructed from shift-invariant families of modules.  
Let $A$ be a (connected $\NN$-) graded ring, generated in degree 1.
A {\em point module} over $A$ is a cyclic graded $A$-module with Hilbert series $1/(1-t)$.  Suppose that (right) $A$-point modules are parameterised by a projective scheme  $X$.
Let the point module corresponding to $x \in X$ be $M^x$.  
Then the shift functor $\Psi:  M \mapsto M[1]_{\geq 0}$ induces an automorphism $\sigma$ of $X$ so that $\Psi(M^x) \cong M^{\sigma(x)}$.  

The following result goes back to \cite{ATV1990} (see also \cite{VdBtranslation}), although in this form it is due to Rogalski and Zhang.

\begin{theorem} \label{thm:RZ}
{\rm (\cite[Theorem 4.4]{RZ2008})}
There is an invertible sheaf $\sL$ on $X$ so that there is a homomorphism $\phi:  A \to B(X, \sL, \sigma)$ of graded rings, where $B(X, \sL, \sigma)$ is the twisted homogeneous coordinate ring defined in \cite{AV}.  If $A$ is noetherian then $\phi$ is surjective in large degree.

The kernel of $\phi$ is equal in large degree to 
	\[ J = \bigcap \{ \Ann_A(M) \st M \mbox{ is a $C$-point module for some commutative $\kk$-algebra $C$ } \}.\]
\end{theorem}

The purpose of this section is to give a  version of this theorem for a (not necessarily connected graded) algebra graded by an arbitrary monoid $\Gamma$.

We first need some notation.  Let $\Gamma$ be a monoid and let $A$ be a $\Gamma$-graded ring.  If $M$ is a $\Gamma$-graded right $A$-module and $\gamma \in \Gamma$, we define the {\em shift} $M(\gamma)$ of $M$ by $\gamma$ as:
\[ M(\gamma) = \bigoplus_{\delta \in \Gamma} M(\gamma)_{\delta},\]
where $M(\gamma)_{\delta} = M_{\gamma \delta}$.
We note that 
\beq \label{foo} M(\gamma)_{\delta} A_{\epsilon} = M_{\gamma\delta} A_{\epsilon} \subseteq M_{\gamma \delta \epsilon} = M(\gamma)_{\delta \epsilon}, \eeq
so $M(\gamma)$ is again a $\Gamma$-graded right $A$-module.
Note that 
\[ (M(\gamma))(\delta)_{\epsilon} = M(\gamma)_{\delta \epsilon} = M_{\gamma \delta \epsilon} = M(\gamma \delta)_{\epsilon}\]
and so $(M(\gamma))(\delta) $ is canonically isomorphic to $ M(\gamma \delta)$.

If $M$ is a left module we define $M(\gamma)_\delta = M_{\delta \gamma}$.  Then \eqref{foo} becomes:
\[ 
A_{\epsilon} M(\gamma)_{\delta} = A_{\epsilon} M_{\delta \gamma} \subseteq M_{\epsilon \delta \gamma} = M(\gamma)_{\epsilon \delta},\]
as needed.
We have
\[ (M(\gamma))(\delta)_{\epsilon} = M(\gamma)_{\epsilon \delta} = M_{\epsilon \delta \gamma} = M(\delta \gamma)_{\epsilon}\]
so $(M(\gamma))(\delta) $ is canonically isomorphic to $  M(\delta \gamma )$.

If $A$ is a $\Gamma$-graded ring,  an {\em intermediate series} module over $A$ is a $\Gamma$-graded left or right $A$-module $M$ so that $\dim M_\gamma = 1$ for all $\gamma \in \Gamma$.  We will  use a shift-invariant family of intermediate series modules to construct a ring homomorphism  from $A$ to a $\Gamma$-graded ring, giving a version of Theorem~\ref{thm:RZ} in this setting.

Our notation for smash products is that if $\Gamma$ acts on $A$ then $A\rtimes \Gamma = \bigoplus_{\gamma \in \Gamma} A t^\gamma$, where $t^\gamma t^\delta = t^{\gamma \delta}$ and $t^\gamma r = r^{\gamma} t^\gamma$ for all $r \in A$, $\gamma \in \Gamma$.

\begin{theorem}\label{thm:right}
Let $\Gamma$ be a monoid with identity $e$ and let $A$ be a $\Gamma$-graded ring.
Let $X$ be a reduced affine scheme that parameterises a set of intermediate series right $A$-modules, in the sense that for $x \in X$ there is a  module $M^x$ with  basis $\{v^x_\gamma \st \gamma \in \Gamma\}$, and that there is a $\kk$-linear function $\phi:  A \to \kk[X]$ so that
\[  v^x_e r = \phi(r)(x) v^x_{\gamma}\]
for all $\gamma \in \Gamma, r \in A_\gamma$.
Further suppose that shifting defines a group antihomomorphism $\sigma:  \Gamma \to \Aut_\kk (X),\gamma \mapsto \sigma^\gamma$ so that $M^x(\gamma) \cong M^{\sigma^\gamma(x)}$.
Here we require that the isomorphism maps $v^x_{\gamma \delta} \mapsto v^{\sigma^\gamma(x)}_{\delta}$.

In this setting the map
\[ \Phi:  A \to \kk[X] \rtimes \Gamma, \quad r\in A_\gamma \mapsto \phi(r) t^\gamma\]
is a graded homomorphism of  algebras.  Further, 
\[ \ker \Phi = \bigcap_{x \in X}  \Ann_A M^x.\]
\end{theorem}

\begin{proof}
Let $\Gamma$ act on $\kk[X]$ by $f^\gamma = (\sigma^\gamma)^*(f)$, so $\sigma$ defines a homomorphism from $\Gamma \to \Aut_\kk(\kk[X])$.

Let $r \in A_\gamma, s \in A_\delta$, and let $\alpha:  V^x(\gamma) \to V^{\sigma^{\gamma}(x)}$ be the given isomorphism.  Then:
\[ \alpha(v^x_{\gamma} s) = v_e^{\sigma^{\gamma}(x)}s= \phi(s)(\sigma^{\gamma}(x)) v^{\sigma^{\gamma}(x)}_\delta = \alpha( \phi(s)(\sigma^{\gamma}(x)) v^x_{ \gamma \delta}).\]
So 
\beq \label{one} v^x_{\gamma} s = \phi(s)^{\gamma}(x) v^x_{ \gamma\delta}. \eeq

Now, using \eqref{one}, we obtain:
\[ \phi(rs)(x) v^x_{\gamma \delta} = 
v^x_e rs = \phi(r)(x) v^x_{\gamma} s  = \phi(r)(x)\phi(s)^{\gamma}(x) v^x_{ \gamma\delta}\]
and so
\beq \label{two} \phi(rs) = \phi(r) \phi(s)^{\gamma}. \eeq
Then by \eqref{two} we have
\[ \Phi(rs) = \phi(rs) t^{\gamma \delta} = \phi(r) \phi(s)^{\gamma}t^{\gamma \delta} = \phi(r) t^\gamma \phi(s)t^\delta = \Phi(r) \Phi(s).\]

Since $\Phi$ is graded, $\ker \Phi$ is a graded ideal of $A$.  
If $r \in A$ is homogeneous then 
\[ \Phi(r) = 0 \iff \phi(r) = 0 \iff v^x_e r = 0 \mbox{ for all } x \in X.\]
Let $\gamma \in \Gamma$.  Then 
\[ v^x_e r = 0 \mbox{ for all } x \in X \iff v^{\sigma^\gamma(x)}_e r = 0 \mbox{ for all } x \in X 
\iff v^x_\gamma r = 0 \mbox{ for all } x \in X,\]
using the isomorphism between $M^x(\gamma)$ and $M^{\sigma^\gamma(x)}$.
So 
\[ \Phi(r)  =0 \iff v^x_\gamma r = 0 \mbox{ for all } x \in X, \gamma \in \Gamma \iff
r  \in \bigcap_{x \in X} \Ann_A M^x.\]

\end{proof}

(The reason we require $X$ in the theorem statement to be reduced is that we are constructing $\Phi$ from the closed points of $X$, and so effectively from the reduced induced structure on $X$.)

\begin{remark}\label{rem:antihom} We need the map $\sigma$ in Theorem~\ref{thm:right} to be an antihomomorphism because of the equations:
\[ M^{\sigma^{\gamma \delta}(x) } \cong M^x(\gamma \delta) = (M^x(\gamma))(\delta) \cong  M^{\sigma^\gamma(x)}(\delta) \cong  M^{\sigma^\delta(\sigma^{\gamma}(x))}.\]
\end{remark}

\begin{remark}\label{rem:right}
There is a universal module $M$ for the family $\{M^x \ | \ x \in X \}$, which is isomorphic as a $\kk[X]$-module to $\bigoplus_{\gamma \in \Gamma} \kk[X] v_\gamma$.  
The module structure is given by
\beq\label{bar} v_\gamma s = \phi(s)^\gamma v_{\gamma \delta} \eeq
for $s \in A_\delta$.
If we consider the natural right action of $A$ on $M = \kk[X] \rtimes \Gamma$ then we have
$t^\gamma \cdot s = t^\gamma \Phi(s) = t^\gamma \phi(s) t^\delta = \phi(s)^\gamma t^{\gamma \delta}$
for $s \in A_\delta$.
This agrees with \eqref{bar} if we set $v_\gamma = t^\gamma$, and so $M \cong \kk[X] \rtimes \Gamma$.
\end{remark}

The theorem for left modules is:
\begin{theorem}\label{thm:left}
Let $\Gamma$ be a monoid with identity $e$ and let $A$ be a $\Gamma$-graded ring.
Let $X$ be a reduced affine scheme that parameterises a set of intermediate series left $A$-modules, in the sense that the left module $N^x$ has a basis $\{v^x_\gamma \st \gamma \in \Gamma\}$ and  that there is a $\kk$-linear function $\phi:  A \to \kk[X]$ so that
\[ r v^x_e = \phi(r)(x) v^x_{\gamma}\]
for all $\gamma \in \Gamma, r \in A_\gamma$.
Further suppose that shifting defines a group homomorphism $\sigma:  \Gamma \to \Aut_\kk (X),\gamma \mapsto \sigma^\gamma$ so that $N^x(\gamma) \cong N^{\sigma^\gamma(x)}$.
Here we require that the isomorphism maps $v^x_{\delta \gamma} \mapsto v^{\sigma^\gamma(x)}_{\delta}$.

In this setting the map
\[ \Phi:  A \to \kk[X] \rtimes \Gamma^{op} \quad r\in A_\gamma \mapsto \phi(r) t^\gamma\]
is a graded antihomomorphism of algebras.  Further, 
\[ \ker \Phi = \bigcap_{x \in X}  \Ann_A N^x.\]
\end{theorem}

\begin{proof}
We repeat the proof above to ensure that the change of notation from right to left is handled correctly.  
Again, let $f^\gamma = (\sigma^\gamma)^* f$, so $\sigma$ defines a homomorphism from $\Gamma^{op} \to \Aut_\kk \kk[X]$.
Let $r \in A_\gamma, s \in A_\delta$, and let $\alpha:  V^x(\delta) \to V^{\sigma^{\delta}(x)}$ be the given isomorphism. 
Then:
\[ \alpha(r v^x_\delta) = r v^{\sigma^\delta(x)}_e = \phi(r)(\sigma^\delta(x)) v^{\sigma^\delta(x)}_{\gamma} = \alpha(\phi(r)(\sigma^\delta(x)) v^x_{\gamma \delta}).\]
So
\beq \label{three}
r v^x_\delta = \phi(r)(\sigma^\delta(x)) v^x_{\gamma \delta}.
\eeq

Now, using \eqref{three}, we obtain:
\[ \phi(rs)(x) v^x_{\gamma \delta} = rs v^x_e = \phi(s)(x) r v^x_{\delta} = \phi(s)(x) \phi(r)(\sigma^\delta(x)) v^x_{\gamma \delta}\]
and so
\beq \label{four} \phi(rs) = \phi(s) \phi(r)^{\delta}. \eeq
Then by \eqref{four} we have
\[ \Phi(rs) = \phi(s) \phi(r)^{\delta}t^{ \gamma \delta} = \phi(s) \phi(r)^{\delta}  t^{\delta \circ_{\op} \gamma }= \phi(s) t^\delta \phi(r) t^\gamma = \Phi(s) \Phi(r).\]

The proof of the last statement is identical to the proof in Theorem~\ref{thm:right}.
\end{proof}

\begin{remark}\label{rem:hom}
We need the map $\sigma$ in Theorem~\ref{thm:left} to be a homomorphism because:
\[ N^{\sigma^{\gamma \delta}(x) } = N^x(\gamma \delta) = (N^x(\delta))(\gamma) = N^{\sigma^\delta(x)}(\gamma) = N^{\sigma^\gamma(\sigma^{\delta}(x))}.\]
Note also that  a graded anti-homomorphism from a $\Gamma$-graded algebra should map to a $\Gamma^{\op}$-graded algebra, as we indeed have.
\end{remark}

\begin{remark}\label{rem:left}
We likewise obtain the universal left module for the $N^x$ from $\Phi$.
Set $N = \kk[X] \rtimes \Gamma^{op}$. 
The left action induced by $\Phi$ is $r \cdot \delta = \delta \Phi(r)$ because $\Phi$ is an anti-homomorphism, so we get
\[ r \cdot t^\delta = t^\delta \Phi(r) = t^\delta \phi(r) t^\gamma = \phi(r)^{\delta} t^{\delta \circ_{op} \gamma }= \phi(r)^\delta t^{\gamma \delta}\]
for $r \in A_{\gamma}$, which is the structure we expect.
\end{remark}

\begin{remark}
Let $\Bir(X)$ be the group of birational self-maps of $X$.  In the settings above, suppose that shifting defines elements of $\Bir(X)$, in the sense that $\sigma$ maps $\Gamma$ to $\Bir(X)$. We get a generalization of Theorems \ref{thm:right} and \ref{thm:left} by replacing $\kk[X]$ and $\Aut(\kk[X])$ with $\kk(X)$ and $\Bir(X)$, respectively. 
\end{remark}

\section{Intermediate series modules over higher rank Witt algebras}\label{sec2}
Let $\Gamma$ be a rank $n$ $\ZZ$-submodule of $\fk$. 
The {\em rank $n$ Witt algebra} $W_\Gamma$ (or {\em higher rank Witt algebra} if $n\geq 2$, sometimes called the centerless higher rank Virasoro algebra) is the  Lie algebra with $\kk$-basis 
$\{e_\nu\mid\nu\in \Gamma\}$ and bracket
\[
[e_\mu,e_\nu]=(\nu-\mu)e_{\nu+\mu}
\]
for $\nu,\mu\in \Gamma$. The rank one Witt algebra is the ``usual'' Witt algebra, which we denote by $W$.

As $U(W_\Gamma)$ is $\Gamma$-graded one can consider intermediate series modules as in Section \ref{sec1}. They are the standard intermediate series modules of Lie algebras, called also Harish-Chandra modules over $(W_\Gamma,\fk e_0)$; i.e., modules of the form $\oplus_{\gamma\in \Gamma} V_{\gamma}$, where $V_{\gamma}$ is the $\gamma$-eigenspace for $ e_0$ and has dimension $1$.

The intermediate series $W_\Gamma$-modules have been classified in \cite[Theorem 2.1]{SuVirasoro}, generalizing the classification \cite{KaplanskySantharoubane} for the Witt algebra. 
There are three families of 
indecomposable intermediate series $W_\Gamma$-modules:
\begin{align*}
V_{(\alpha,\beta)}&=
\oplus_{\nu\in \Gamma} \fk v_\nu,\quad e_\mu v_\nu=(\alpha+\beta \mu+\nu)v_{\mu+\nu},\\
A_{(\alpha,\beta)}&=
\oplus_{\nu\in \Gamma}\fk v_\nu,\quad e_\mu v_\nu=\begin{cases}
\nu v_{\mu+\nu}&  \nu\neq 0,\,\mu+\nu\neq 0,\\ (\alpha+\beta\mu)v_{\mu}&\nu=0,\\ 0&\mu+\nu=0,
\end{cases}\\
B_{(\alpha,\beta)}&=
\oplus_{\nu\in \Gamma}\fk v_\nu,\quad e_\mu v_\nu=\begin{cases}
(\mu+\nu) v_{\mu+\nu}&  \nu\neq 0,\,\mu+\nu \neq 0,\\0&\nu=0,\\ (\alpha+\beta\mu)v_{0}&\mu+\nu=0,
\end{cases}\\
\end{align*}
where $(\alpha,\beta)\in \AA^2$. 
Note that $A_{(\alpha,\beta)}$, $B_{(\alpha,\beta)}$ are only defined where $(\alpha, \beta) \neq (0,0)$ and depend up to isomorphism (rescaling of $v_0$) only on $[\alpha:\beta]\in \PP^1$. We will therefore denote them by $A_{[\alpha:\beta]}$, $B_{[\alpha:\beta]}$.
Note also that we have $A_{[1:0]} \cong V_{(0,1)}$ (by $v_0\mapsto v_0$ and $v_\nu\mapsto \nu v_\nu$ when $\nu\neq 0$) and $B_{[1:0]} \cong V_{(0,0)}$ (by $v_0\mapsto \nu v_0$ and $v_\nu\mapsto v_\nu$ when $\nu\neq 0$). 

\begin{remark}
Note that $A_{[\alpha:\beta]}$ contains a simple submodule $\oplus_{0\neq \nu\in\Gamma}\fk v_{\nu}$ with a $1$-dimensional trivial quotient.  
On the other hand, $B_{[\alpha:\beta]}$ has the $1$-dimensional trivial submodule $\fk v_\nu$, and the quotient is a simple module. 
This is explained by the isomorphism $B_{[\alpha:\beta]}'\cong A_{[\alpha:\beta]}$, where $'$ denotes the adjoint.  
(If $M=\oplus_{\gamma\in\Gamma}\fk v_\gamma$ is a left $\Gamma$-graded $W_\Gamma$-module, the {\em adjoint} (or {\em restricted dual})   of  $M$ is the left $\Gamma$-graded $W_\Gamma$-module $M'$ with $M'_\gamma=\Hom_\kk(M_{-\gamma},\kk)$, $v'_\gamma=v^*_{-\gamma}$, and $e_\mu v'_\gamma=-v^*_{-\gamma}e_{\mu}$.) 
\end{remark}

\begin{remark}
We use a slightly different presentation of the families $A_{[\alpha:\beta]}$, $B_{[\alpha:\beta]}$ than in \cite{SuVirasoro}. In loc.cit the last two families are replaced by 
$\tilde{A}({a'})$ defined by
\[
e_\mu v'_\nu=(\nu+\mu)v'_{\mu+\nu},\; \nu\neq 0,\quad e_\mu v_0=\mu(1+(\mu+1)a')v'_\mu,
\]
and  by $\tilde{B}(a')$ defined by
\[
e_\mu v'_\nu=\nu v'_{\mu+\nu},\; \nu\neq -\mu,\quad e_\mu v'_{-\mu}=-\mu(1+(\mu+1)a')v'_0,
\]
for $a'\in \fk\cup\{\infty\}$. If $a'=\infty$ then $1+(\mu+1)a'$ in the above definition is regarded as $\mu+1$. 
Note that $\tilde{A}(a')$ (resp. $\tilde{B}(a')$) is isomorphic to $A_{[1+a':a']}$ (resp. $B_{[1+a':a']}$) if $a'\neq \infty$ and to $A_{[1:1]}$ (resp. $B_{[1:1]}$) if $a'=\infty$, for $v_\nu=\nu v'_\nu$ (resp. $v_\nu=\frac{1}{\nu} v'_\nu$) if $\nu\neq 0$, and $v_0=v'_0$.

For the Witt algebra the choice of the basis is the same in \cite{KaplanskySantharoubane}, however there $a'\in \fk$ and modules are classified up to inversion:   replacing $v_{\nu}$ by $-v_{-\nu}$.

\end{remark}

Let us show how to obtain the intermediate series modules using results of Section \ref{sec1}. 

\begin{proposition}\label{prop:definemap}
Let $\Gamma$ act on $\fk[a,b]$ as $t^\nu. p(a,b)=p(a+\nu,b)t^\nu$, and let $T := \kk[a,b] \rtimes \Gamma$. 
The map $\phi:W_\Gamma\to T$, $\phi(e_\mu)=(a+b\mu)t^\mu$, induces an anti-homomorphism $\Phi:U(W_\Gamma)\to T$. Consequently, $T$ is a left $U(W)$-module via 
$e_\mu.p(a,b)t^\nu=(a+\nu +b\mu) p(a,b) t^{\mu+\nu}$.
\end{proposition}

\begin{proof}
Note that $\mathbb{A}^2$ parametrises a set 
of intermediate series modules 
$N^{(\alpha,\beta)}:=V_{(\alpha,\beta)}$ and 
$e_{\mu}v_0^{(\alpha,\beta)}=(a+b\mu)((\alpha,\beta)) v_{\mu}^{(\alpha,\beta)}$. Further, $N^{(\alpha,\beta)}(\nu)\cong N^{(\alpha+\nu,\beta)}$  and hence $\sigma^\nu((\alpha,\beta))=(\alpha+\nu,\beta)$ (using the notation of Section \ref{sec1}). 
The proposition therefore follows by Theorem \ref{thm:left} and Remark~\ref{rem:left}. 
\end{proof}

\begin{remark}\label{rem:samemap}
Let $\Gamma= \ZZ$ and $T = \kk[a,b] \rtimes \ZZ$.
We may compose the map $\Phi$ of Proposition~\ref{prop:definemap} with the canonical anti-automorphism $e_n \mapsto -e_{n}$ of $U(W)$ to obtain a homomorphism $\Phi':  U(W) \to T, e_n \mapsto (-a-bn) t^n$.

Recall that in \cite{SierraWalton2} a homomorphism $\hat{\phi}$ was constructed from $U(W)$ to
\[ T' := \kk\ang{u,v,v^{-1},w}/(uv-vu-v^2, uw-wu-wv, vw-wv),\]
defined by $\hat{\phi}(e_n) = (u-(n-1)w)v^{n-1}$.
The reader may verify that $\alpha: T' \to T$ defined by
\[ u \mapsto (b-a)t, \quad v \mapsto t, \quad w \mapsto bt\]
is an isomorphism of graded rings and that $\alpha \hat{ \phi} = \Phi'$.
Thus Proposition~\ref{prop:definemap} generalises the construction of $\hat{\phi}$.
\end{remark}

We now discuss applications of $\Phi$ to the representation theory of $W_\Gamma$.  
For $p=(\alpha,\beta)\in \AA^2$ we denote by $I(p) $ the ideal $(a-\alpha,b-\beta)$ in $\fk[a,b]$. 
For $q$ infinitely near to $p$, corresponding to $[x:y]\in \PP^1$, we denote by $I(q)$ the ideal $(y(a-\alpha)-x(b-\beta),(a-\alpha)^2,(a-\alpha)(b-\beta),(b-\beta)^2)$. 

Let $B = \Phi(U(W_\Gamma))$, and note that $B $ is contained in the double idealizer $R=\kk[a,b] + ( I(0,0) T \cap TI(0,1))$.  
From the discussion in the introduction, then, we expect three families of intermediate series $U(W_\Gamma)$-modules, one parameterised by
$\AA^2 \ssm\{(0,0), (0,1)\}$ and two parameterised by $\PP^1$. 
Note that because $\Phi$ is an anti-homomorphism, {\em right} $B$-modules will correspond to {\em left} $U(W_\Gamma)$-modules.

By construction of $\Phi$ we have $V(\alpha,\beta)\cong T/ I(p)T$, considered as a $B$-module.  
Removing $V(0,0)$ and $V(0,1)$ we obtain the two-dimensional family we expect.  
We next show that we also obtain the two $\PP^1$-families $A_{[\alpha:\beta]}$ and $B_{[\alpha:\beta]}$.

\begin{proposition}\label{prop:P}
Let $[x:y] \in \PP^1$ and let $I(q) = (ya-xb,a^2, ab, b^2)$ define a point infinitely near to $(0,0)$.
Let 
\[ P(q) = \frac{\fk[a,b] + I{(0,0)}T}{I{(0,0)} + I(q)T}.\]
Then
$A_{[x:y]}\cong P(q)$. 
 \end{proposition}

\begin{proof}
If $w \in \fk[a,b] + I{(0,0)}T$ let $\ol{w}$ be the  image of $w$ in $P(q)$. 
If $x\neq 0$ we choose a basis 
\[
v_\nu=\begin{cases}
\ol{at^{\nu}}&{\nu\neq 0},\\
 \bar 1&{\nu=0}
 \end{cases}
\]
for $P(q)$.

Using the anti-homomorphism, we compute for $\nu\neq 0$
\[
e_\mu.v_\nu=\ol{at^\nu(a+b\mu)t^\mu}=\ol{a(a+b\mu+\nu)t^{\mu+\nu}}=\nu \ol{at^{\mu+\nu}}=
\begin{cases}
\nu v_{\nu+\mu} & \nu+\mu\neq 0,\\
0&\nu+\mu=0.
\end{cases}
\]
and 
\[
e_\mu.v_0=\ol{(a+b\mu)t^\mu}=\ol{\left(a+\frac{y}{x}a\mu\right)t^\mu}=\left(1+\frac{y}{x}\mu\right)v_\mu,
\]
so $P(q) \cong A_{[x:y]}$ as claimed.

If $y\neq 0$ we pick a basis 
\[
v_\nu=\begin{cases}
\ol{bt^{\nu}}&{\nu\neq 0},\\
 \ol{1}&{\nu=0},
 \end{cases}
\]
and obtain 
$e_\mu.v_{\nu}=\nu v_{\nu+\mu}$, 
$e_\mu.v_0=(\frac{x}{y}+\mu)v_\mu$,
$e_\mu.v_{-\mu}=0$.
Thus $P(q) \cong A_{[x:y]}$ again.
\end{proof}

In the next result, note the change  of sides from the left modules $Q(q)$ defined in the introduction.

\begin{proposition}\label{prop:Q}
Let $[x:y] \in \PP^1$ and let $I(q) = (ya-x(b-1),a^2, a(b-1), (b-1)^2)$ define a point infinitely near to $(0,1)$.
Let 
\[ Q(q) = \frac{I(0,1) \oplus  \bigoplus_{0\neq\nu\in \Gamma}\fk[a,b]t^\nu}{I{(q)} \oplus \bigoplus_{0\neq \nu\in \Gamma}I{(0,1)}t^\nu}.\]
Then $B_{[x:y]}\cong Q(q)$.
\end{proposition}

\begin{proof}
If $x\neq 0$ we choose a basis 
\[
v_\nu=\begin{cases}
{\ol{t^{\nu}}}&{\nu\neq 0},\\
 \ol{ a}&{\nu=0}
 \end{cases}
\]
for $Q(q)$.
We compute for $\nu+\mu\neq 0$, $\nu\neq 0$
\[
e_\mu.v_\nu=\ol{(a+b\mu+\nu)t^{\mu+\nu}}=(\mu+\nu) \ol{t^{\mu+\nu}}=(\mu+\nu) v_{\mu+\nu}
\]
and 
\[
e_\mu.v_0=\ol{a(a+b\mu)t^\mu}=0,\quad e_\mu.v_{-\mu}=\ol{a+b\mu-\mu}=\left(1+\frac{y}{x}\mu\right)v_0.
\]

If $y\neq 0$ we pick a basis 
\[
v_\nu=\begin{cases}
{\nu \ol{t^{\nu}}}&{\nu\neq 0},\\
 \ol{b}&{\nu=0}.
 \end{cases}
\]
We get 
$e_\mu.v_\nu=\nu v_{\mu+\nu}$, $e_\mu.v_0=0$, $e_\mu.v_{-\mu}=\left(\frac{x}{y}+\mu\right)v_0$.
\end{proof}

\section{Factors of $U(W_\Gamma)$}
In this section we  generalise techniques from \cite{SierraWalton2} to show that $B=\Phi(U(W_\Gamma))$ is not left or right noetherian.
This in particular implies that $U(W_\Gamma)$ is not left or right noetherian, which  was  proved earlier in \cite{SierraWalton1,SierraWalton2}.

For  $0\neq \mu\in \Gamma$, let 
\[p_\mu=e_\mu e_{3\mu}-e_{2\mu}^2-\mu e_{4\mu}.\]

\begin{lemma}\label{lem:pmu}
We have $\Phi(p_\mu)=\mu^2b(1-b)t^{4\mu}$.
\end{lemma}

\begin{proof}
Let us compute 
\begin{align*}
\Phi(e_\mu e_{3\mu}-e_{2\mu}^2-\mu e_{4\mu})&=
\left((a+3\mu b)(a+\mu b+3\mu)-(a+2\mu b)(a+2\mu b+2\mu)-\mu(a+4\mu b)\right)t^{4\mu}\\
&=\mu^2b(1-b)t^{4 \mu}.
\end{align*}
\end{proof}

Fix $0 \neq \mu \in \Gamma$ and let $I = B \Phi(p_\mu) B$. 

\begin{lemma}\label{all}
For all $\nu\in \Gamma$ we have $b(1-b)t^\nu\in I$. 
 In particular, $I$ does not depend on the choice of $\mu$.
 Consequently, $I=b(1-b)\fk[a,b]\rtimes \Gamma$.{}
\end{lemma}

\begin{proof}
We have
\[
\Phi(e_{\nu-4\mu})b(1-b)t^{4\mu}-b(1-b)t^{4\mu}\Phi(e_{\nu-4\mu})=
(\Phi(e_{\nu-4\mu})-\Phi(e_{\nu-4\mu})-4\mu)b(1-b)t^{\nu}=-4\mu b(1-b)t^\nu.
\]
Thus the first claim follows by Lemma \ref{lem:pmu}. 
Note that $I\subseteq b(1-b)\fk[a,b]\rtimes \Gamma$, and as $b(1-b)\in I$ and $a\in B$, we have $b(1-b)\fk[a]\rtimes \Gamma\subseteq I$. Since also $(a+b\mu)t^{\mu}\in B$, we easily obtain  by induction on $n$ that $b(1-b)b^n\fk[a]\rtimes \Gamma\subseteq I$ for all $n\geq 0$, and thus the last claim.
\end{proof}

\begin{proposition}\label{prop:noetherian}
The ideal $I$ is not finitely generated as a left or right ideal of $B$.
\end{proposition}

\begin{proof}
We first compute
\begin{multline}\label{left}
(a+b\nu_1)t^{\nu_1}\cdots (a+b\nu_l)t^{\nu_l}p(a,b)b(1-b)t^{\lambda}=\\
(a+b\nu_1)\cdots (a+b\nu_l+\nu_1+\cdots+\nu_{l-1})p(a+\nu_1+\cdots+\nu_{l-1}+\nu_l,b)b(1-b)t^{\nu_1+\cdots+\nu_l+\lambda},
\end{multline} 
\begin{multline}\label{right}
p(a,b)b(1-b)t^{\lambda}(a+b\nu_1)t^{\nu_1}\cdots (a+b\nu_l)t^{\nu_l}=\\
p(a,b)b(1-b)(a+b\nu_1+\lambda)\cdots (a+b\nu_l+\lambda+\nu_1+\cdots+\nu_{l-1})t^{\lambda+\nu_1+\cdots+\nu_l}.
\end{multline}

Let us assume that $I$ is finitely generated as a left ideal of $B$.
Then there exist $\mu_1,\dots,\mu_k\in \Gamma$ such that $I=B(I_{\mu_1}+\cdots +I_{\mu_k})$. 
Let us take $\mu\neq \mu_i$, $1\leq i\leq k$. 
It follows from \eqref{left} that $(B(I_{\mu_1}+\cdots +I_{\mu_k}))_\mu$ is contained in $(a,b)b(1-b)t^\mu$, a contradiction to Lemma \ref{all}.

Let us assume now that $I$ is finitely generated as a right ideal in $B$.
Then there exist $\mu_1,\dots,\mu_k\in \Gamma$ such that $I=(I_{\mu_1}+\cdots +I_{\mu_k})B$. 
For $\mu\neq \mu_i$, $1\leq i\leq k$, we obtain from \eqref{right} that 
$((I_{\mu_1}+\cdots +I_{\mu_k})B)_\mu$ is contained in $(a+\mu,b-1)b(1-b)t^\mu$, which again contradicts Lemma \ref{all}. 
\end{proof}

\begin{remark}
Note that the same proof works if $\Gamma$ is a submonoid of $\fk$. 
Lemma \ref{all} yields in this case $b(1-b)t^{n\mu}\in I$, for $n\geq 4$. The proof of Proposition \ref{prop:noetherian}  can then be adapted in an obvious way to apply to this a slightly more general situation. In particular, $\Phi(U(W_+))$ is not noetherian, where $W_+$ is the subalgebra of $W$ generated by $\{e_n : n \geq 1\}$.  (This last statement is proved in \cite{SierraWalton2}
\end{remark}

We now show that the image $B_\beta$ of the map $\phi_\beta:U(W)\to B/(b-\beta)$ induced from $\Phi$ is noetherian for every $\beta\in \fk$. This is an analogue of \cite[Proposition 2.1]{SierraWalton2}.

\begin{lemma}\label{phi}
We have $B_0\cong \fk+a(\fk[a]\rtimes \Gamma)$, $B_1\cong \fk+(\fk[a]\rtimes \Gamma)a$, $B_\beta\cong \fk[a]\rtimes \Gamma$ for $\beta\neq 0,1$.
\end{lemma}

\begin{proof}
The lemma is obvious for $\beta=0,1$. 
Assume therefore that $\beta\neq 0,1$.
Let us compute
\[(a+\beta \mu)t^{\mu}(a+\beta\nu)t^{\nu}-a(a+\beta(\mu+\nu))t^{\mu+\nu}=
(\mu a+\beta\mu(\beta\nu+\mu))t^{\mu+\nu}\in B_\beta.
\]
Subtracting $\mu(a + b (\mu+\nu) t^{\mu+\nu}$, we thus have $\beta\mu\nu(\beta-1)t^{\mu+\nu}\in B_\beta$, and hence our claim.
\end{proof}

\begin{proposition}\label{prop:noetherian2}
$B_\beta$ is noetherian for every $\beta\in \fk$.
\end{proposition}

\begin{proof}
For $\beta\neq 0,1$ this follows by \cite[Theorem 4.5]{McConnellRobson} using Lemma \ref{phi}. 
Let us note that $B_0\cong B_1$ by conjugation with $a$. 
It thus suffices to prove that $B_0$ is right noetherian and $B_1$ is left noetherian. We show that $B_0$ is right noetherian, and following the same argument one can show that $B_1$ is left noetherian. 

We first note that $I=a(\fk[a]\rtimes\Gamma)$ is a maximal right ideal in $C=\fk[a]\rtimes \Gamma$. To see this, let $J\neq I$ be a right ideal which contains $I$. 
Take an element $c=\sum \alpha_\mu t^\mu\neq 0$ in $J$ with the minimal number of nonzero coefficients. 
Since $ca=\sum \alpha_\mu(a+\mu)t^\mu\in J$ and hence $\sum \alpha_\mu \mu t^\mu\in J$, the minimality assumption implies  that $J=\fk[a]\rtimes \Gamma$.

The proposition now follows by \cite[Theorem~2.2]{Robson} using Lemma \ref{phi}. 
\end{proof}

\begin{remark}\label{rem:primitive}
We remark that for any  $\beta$ the modules $V(\alpha, \beta)$ are all faithful over $B_\beta$, and it follows easily that the $B_\beta$ are primitive.  
In general, the primitive factors of $U(W_\Gamma)$ are unknown, even for $\Gamma = \ZZ$.
\end{remark}


\end{document}